\documentclass[mathpazo]{cicp}

\usepackage{epsfig,amsfonts,amsmath,amssymb}
\usepackage[usenames,dvipsnames]{color}
\usepackage[IT]{subfigure}
\usepackage{amsthm}
\usepackage{mathrsfs}

\usepackage[T1]{fontenc}
\usepackage[utf8]{inputenc}
\usepackage[english]{babel}

\newtheorem{thm}{Theorem}
\newtheorem{lem}[thm]{Lemma}
\newtheorem{cor}[thm]{Corollary}

\theoremstyle{definition}

\newtheorem{as}{Assumption}

\theoremstyle{remark}
\newtheorem{rem}{Remark}

\newcommand{\D}{\mathcal{D}}
\newcommand{\HH}{\mathcal{H}}
\newcommand{\RR}{\mathcal{R}}
\newcommand{\C}{\mathscr{C}}

\newcommand{\e}{\mathrm{e}}

\newcommand{\inv}{^{-1}}
\newcommand{\dede}[2]{\frac{\partial #1}{\partial #2}}
\newcommand{\dedex}{\dede{}{x}}
\newcommand{\dedey}{\dede{}{y}}
\newcommand{\dd}[2]{\frac{\mathrm{d} #1}{\mathrm{d} #2}}

\newcommand{\bbm}{\begin{bmatrix}}
\newcommand{\ebm}{\end{bmatrix}}
\newcommand{\R}{\mathbb{R}}

\def\norm#1{\lVert#1\rVert}
\def\normL2#1{\norm{#1}}
\def\normH1#1{\norm{#1}_{H^1(\Omega)}}

\begin{document}
%%%%% title : short title may not be used but TITLE is required.
% \title{TITLE}
% \title[short title]{TITLE}
\title{A full space-time convergence order analysis of operator splittings for linear dissipative evolution equations}

%%%%% author(s) :
%\author[E.~Hansen, E.~Henningsson]{Eskil Hansen\affil{1} and Erik Henningsson\affil{1}\comma\corrauth}
%\address{\affilnum{1}\ Centre for Mathematical Sciences, Lund University, P.O.\ Box 118, SE-221 00 Lund, Sweden}
%\emails{{\tt eskil@maths.lth.se} (E.~Hansen),
%{\tt erikh@maths.lth.se} (E.~Henningsson)}
\author[E.~Hansen, E.~Henningsson]{Eskil Hansen and Erik Henningsson\corrauth}
\address{Centre for Mathematical Sciences, Lund University, P.O.\ Box 118, SE-221 00 Lund, Sweden}
\emails{{\tt eskil@maths.lth.se} (E.~Hansen),
{\tt erikh@maths.lth.se} (E.~Henningsson)}
% multiple authors:
% Note the use of \affil and \affilnum to link names and addresses.
% The author for correspondence is marked by \corrauth.
% use \emails to provide email addresses of authors
% e.g. below example has 3 authors, first author is also the corresponding
%      author, author 1 and 3 having the same address.
% \author[Zhang Z R et.~al.]{Zhengru Zhang\affil{1}\comma\corrauth,
%       Author Chan\affil{2}, and Author Zhao\affil{1}}
% \address{\affilnum{1}\ School of Mathematical Sciences,
%          Beijing Normal University,
%          Beijing 100875, P.R. China. \\
%           \affilnum{2}\ Department of Mathematics,
%           Hong Kong Baptist University, Hong Kong SAR}
% \emails{{\tt zhang@email} (Z.~Zhang), {\tt chan@email} (A.~Chan),
%          {\tt zhao@email} (A.~Zhao)}
% \footnote and \thanks are not used in the heading section.
% Another acknowlegments/support of grants, state in Acknowledgments section
% \section*{Acknowledgments}

%%%%% Begin Abstract %%%%%%%%%%%
\begin{abstract}
The Douglas--Rachford and Peaceman--Rachford splitting methods are common choices for temporal discretizations of evolution equations. In this paper we combine these methods with spatial discretizations fulfilling some easily verifiable criteria. In the setting of linear dissipative evolution equations we prove optimal convergence orders, simultaneously in time and space. We apply our abstract results to dimension splitting of a 2D diffusion problem, where a finite element method is used for spatial discretization. To conclude, the convergence results are illustrated with numerical experiments.
\end{abstract}
%%%%% end %%%%%%%%%%%

%%%%% AMS/PACs/Keywords %%%%%%%%%%%
%\pac{}

%%%% maketitle %%%%%
\maketitle

\section{Introduction}
We consider the linear evolution equation
\begin{equation}
\dot u = Lu = (A+B)u, \quad u(0) = \eta,
\label{eq:evolution}
\end{equation}
where $L$ is an unbounded, dissipative operator. Such equations are commonly encountered in the natural sciences, e.g.\ when modeling advection-diffusion processes. Splitting methods are widely used for temporal discretizations of evolution equations. The competitiveness of these methods is attributed to their separation of the flows generated by $A$ and $B$. In many applications these separated flows can be more efficiently evaluated than the flow related to $L$; a prominent example being that of dimension splitting. We refer to \cite{Hairer,Hundsdorfer,McLachlan} for general surveys.

In the present paper we consider the combined effect of temporal and spatial discretization when the former is given by either the Douglas--Rachford scheme
\begin{equation}
S = (I-kB)\inv(I-kA)\inv(I+k^2AB),
\label{eq:DR}
\end{equation}
or the Peaceman--Rachford scheme
\begin{equation}
S = (I-\frac k2B)\inv(I+\frac k2A)(I-\frac k2A)\inv(I+\frac k2B).
\label{eq:PR}
\end{equation}   
Here, $S$ denotes the operator that takes a single time step of size $k$. Thus $S^n\eta$ constitutes a temporal splitting approximation at time $t = nk > 0$ of the solution $u(t) = \e^{tL}\eta$ of Eq.~\eqref{eq:evolution}. The first order Douglas--Rachford scheme can be constructed as a modification of the simple Lie splitting resulting in an advantageous error structure. An exposition is given in~\cite{Hansen2014}. The Peaceman--Rachford scheme was introduced in \cite{Peaceman} as a dimension splitting of the heat equation. A temporal convergence order analysis of the scheme for linear evolution equations is given in \cite{Hansen2008}, which also features an application to dimension splitting. Convergence orders in time are proven in \cite{Descombes2003,Hansen2013} for nonlinear operators $B$ under various assumptions on the nonlinearity. See also \cite{Schatzman} for further stability considerations.

In the general setting the operators $A$, $B$, and $L$ are infinite dimensional. Therefore, to define an algorithm that can be implemented, any numerical method must replace these operators, that is, a spatial discretization is needed. In our abstract analysis, we consider any spatial discretization fulfilling some assumptions ensuring convergence for the stationary problem. When both a temporal and a spatial discretization has been employed to Eq.~\eqref{eq:evolution} we refer to it as being \emph{fully discretized}. Under similar assumptions to ours, convergence orders are proven in \cite{Crouzeix,Thomee} for full discretizations where implicit Euler or Crank--Nicholson is used as temporal discretization.

Our abstract analysis is applied to dimension splitting combined with a finite element method. As is usually done in practice, we consider spatial discretizations where the finite element matrices are constructed with the help of numerical quadrature schemes. We will refer to these discretization methods as \emph{quadrature finite element methods}. Convergence order analyses for quadrature finite element methods are carried out for linear elliptic PDEs in \cite{Ciarlet,Ciarlet1972,Strang,Thomee} and when they are used as spatial semi-discretizations for a linear parabolic problem in \cite{Raviart}. Full discretizations of a nonlinear parabolic PDE, where the spatial discretization is given by quadrature finite elements, are considered in \cite{Nie}. There, convergence orders are derived when explicit Euler, implicit Euler or a modified Crank--Nicholson method is used as temporal discretization. 

Earlier results about the combined effects of splitting methods and spatial discretizations include the recent paper \cite{Batkai2012}. There,  convergence without orders is proven for full discretizations when exponential splittings are used for temporal discretization of the abstract evolution equation \eqref{eq:evolution}. Full space-time convergence studies for semi-implicit methods applied to various semilinear evolution equations can be found in \cite{Akrivis,Larsson1992,Thomee}. A partial error analysis for the Peaceman--Rachford scheme with orders in time is carried out in \cite{Hundsdorfer1989}. 

However, to our knowledge, there is no abstract convergence analysis, providing optimal orders both in time and space, for full discretizations of Eq.~\eqref{eq:evolution} which only assumes that $L = A+B$ is dissipative and where splitting methods are used as temporal discretization. The aim of this paper is therefore to analyze convergence orders for the splitting schemes \eqref{eq:DR} and \eqref{eq:PR} combined with converging spatial discretizations. We strive to assume regularity only on the initial data $\eta$ in order to make the assumptions easy to verify. 

Our proof will follow in the spirit of \cite{Crouzeix,Thomee}. To this end we analyze the spatial semi-discretization of Eq.~\eqref{eq:evolution} in Section~\ref{sec:spatial_discretization} and then we expand the analysis to full discretizations in Section~\ref{sec:full_discretization}. The analysis of the temporal error is performed in the finite dimensional subspace defined by the spatial semi-discretization. The advantage of this approach is that we need no assumptions on the operators $A$ and $B$ and their relation to $L$. In Section~\ref{sec:dimension_splitting} we present how dimension splitting combined with quadrature finite elements can be fitted into our abstract framework. Our theoretical results are exemplified in Section~\ref{sec:numerical_experiments} with some numerical experiments.

\section{Spatial discretization} 
\label{sec:spatial_discretization}
Let $L: \D(L) \subset \HH \rightarrow \HH$ be a linear, unbounded operator on a real Hilbert space $\HH$. Denote the inner product on $\HH$ by $(\cdot,\cdot)$ and the induced norm by $\norm{\cdot}$. The latter notation is also used for the related operator norm. Throughout the paper $C$ is a generic constant taking different values at different occurrences.
A linear operator $E:\D(E)\subset\HH \rightarrow \HH$ is called maximal dissipative if
\[ (Ev,v) \leq 0, \text{ for all } v \in \D(E), \quad \text{and} \quad  \RR(I-kE) = \HH, \text{ for all } k > 0. \]
Recall that this implies that $E$ generates a strongly continuous semigroup of contractions $\{\e^{tE}\}_{t \geq 0}$ and that the resolvent $(I-kE)\inv$ is nonexpansive on $\HH$ for all $k > 0$, \cite[Theorems 1.3.1 and 1.4.3]{Pazy}. We consider operators exhibiting this property.

\begin{as} \label{as:continuous}
The operator $L:\D(L)\subset\HH \rightarrow \HH$ is maximal dissipative and (for the sake of simplicity) invertible.
\end{as}

As spatial discretization consider a family of finite dimensional subspaces of $\HH$, denoted by $\{\HH_h\}_{0 < h \leq h_{max}}$, which are of increasing dimension as $h$ tends to zero. Equip each of them with its own inner product $(\cdot,\cdot)_{h}$. On these spaces define the discrete operators $A_h:\HH_h\rightarrow \HH_h$, $B_h:\HH_h\rightarrow \HH_h$ and $L_h = A_h + B_h$. The ODE
\begin{equation}
\dot u_h = L_h u_h = (A_h + B_h)u_h, \quad u_h(0) = \eta_h,
\label{eq:discrete_evolution}
\end{equation}
where $\eta_h\in \HH_h$ is an approximation of $\eta$, is then the spatial semi-discretization of the evolution equation \eqref{eq:evolution}. We choose the spaces $\HH_h$, the inner products $(\cdot,\cdot)_{h}$, and the discrete operators such that Assumption~\ref{as:space} is fulfilled.

\pagebreak
\begin{as} \label{as:space}
For fixed $s > 0$ and $q = 0$ or $1$, assume the following: 
\begin{enumerate}
	\item \label{ass:norm_equiv} 
	The norms $\norm{\cdot}$ and $\norm{\cdot}_{h}$ are uniformly equivalent on $\HH_h$, that is
	\[ C_1 \norm{v_h} \leq \norm{v_h}_h \leq C_2 \norm{v_h}, \quad \text{ for all } v_h \in \HH_h, \]
	where the two constants $C_1$ and $C_2$ are independent of $h$.
	
	\item \label{ass:projection} 
	There is a mapping $P_h:\D(P_h)\subset \HH \rightarrow \HH_h$ such that $\D(L^q) \subset \D(P_h)$ and
	\[\norm{P_h v - v} \leq Ch^s \sum_{i=0}^q{\norm{L^i v}}, \quad \text{ for all } v\in \D(L^q),\]
	for a constant $C$ independent of $h$. 
	
	\item \label{ass:discrete_dissipativity}
	For all $h\in (0,h_{max}]$ the operators $A_h$ and $B_h$ are dissipative on $(\HH_h,(\cdot,\cdot)_{h})$.
	
	\item \label{ass:discrete_invetibility} 
	For the sake of simplicity assume that for all $h\in (0,h_{max}]$ the operator $L_h$ is invertible and its inverse is bounded uniformly in $h$.
	
	\item \label{ass:elliptic_error} 
	There is a constant $C$, independent of $h$, such that 
	\[ \norm{L^{-1}v - L_h^{-1} P_h v} \leq Ch^s \sum_{i=0}^q{\norm{L^i v}}, \quad \text{ for all } v \in \D(L^q). \]
\end{enumerate}
\end{as}
\begin{rem} \label{rem:L_h_dissipative}
The operator $L_h$ is dissipative as a direct consequence of Assumption~\ref{as:space}.\ref{ass:discrete_dissipativity}. All the discrete operators $A_h, B_h$ and $L_h$ are also maximal since they are dissipative on the finite dimensional space $\HH_h$.
\end{rem}

We aim to bound the error of the spatial semi-discretization, i.e.\ the difference between the solutions $u(t) = \e^{tL}\eta$ of Eq.~\eqref{eq:evolution} and $u_h(t) = \e^{tL_h}\eta_h$ of Eq.~\eqref{eq:discrete_evolution}. To this end we define the operator
\begin{equation}
Q_h = L_h\inv P_h L: \D(L^{q+1}) \rightarrow \HH_h.
\label{eq:Q_h}
\end{equation}
\begin{lem}\label{lem:space_error}
If Assumptions \ref{as:continuous} and \ref{as:space} are valid, $\eta \in \D(L^{q+2})$, and $\eta_h \in \HH_h$, then
\begin{equation*}
\norm{\e^{tL}\eta - \e^{tL_h}\eta_h} \leq C (\norm{\eta-\eta_h} + h^s \sum_{i=1}^{q+2}{\norm{L^i\eta}}),
\end{equation*}
where $C$ can be chosen uniformly on bounded time intervals and, in particular, independently of $h$ and $n$.
\end{lem}

\begin{proof}
We will repeatedly need the bounds 
\begin{align}
& \norm{(I-Q_h)v} = \norm{(L\inv - L_h\inv P_h) Lv} \leq Ch^s \sum_{i=1}^{q+1}{\norm{L^i v}}, \label{eq:I_Qh}\\
& \norm{(P_h-Q_h)v} \leq \norm{(P_h-I)v} + \norm{(I-Q_h)v} \leq Ch^s \sum_{i=0}^{q+1}{\norm{L^i v}}, \label{eq:Ph_Qh}
\end{align}
which follow from Assumptions \ref{as:space}.\ref{ass:projection} and \ref{as:space}.\ref{ass:elliptic_error} for $v \in \D(L^{q+1})$. Splitting the spatial error into two terms gives
\[ \e^{tL}\eta - \e^{tL_h}\eta_h = (I-Q_h)\e^{tL}\eta + (Q_h\e^{tL}\eta - \e^{tL_h}\eta_h) = \rho(t) + \theta(t). \]
The first term is bounded by Eq.~\eqref{eq:I_Qh}, i.e.
\begin{equation*}
\norm{\rho(t)} \leq Ch^s \sum_{i=1}^{q+1}{\norm{L^i\e^{tL}\eta}} \leq Ch^s \sum_{i=1}^{q+1}{\norm{L^i\eta}}. 
\end{equation*}
Since $L$ and $L_h$ generates strongly continuous semigroups and $\eta \in \D(L^{q+2})$ we get from \cite[Theorem~1.2.4]{Pazy} that $\theta \in C^1([0,T],\HH_h)$ and
\[ \dot\theta(t) = Q_h \e^{tL}L\eta - \e^{tL_h}L_h\eta_h.\]
Therefore we can write 
\[ \dot\theta(t) - L_h\theta(t) = (Q_h-P_h) \e^{tL}L\eta.\]
Testing with $\theta(t)$ we get from the dissipativity of $L_h$ that
\[ \begin{aligned}
\norm{\theta(t)}_{h} \dd{}{t} \norm{\theta(t)}_{h}  &=  (\dot\theta(t), \theta(t))_{h} \\
&\leq  ( (Q_h-P_h) \e^{tL}L\eta, \theta(t))_{h} \\
		&\leq  \norm{(Q_h-P_h) \e^{tL}L\eta}_{h} \norm{\theta(t)}_{h}.
\end{aligned} \]
From the uniform equivalence of norms on $\HH_h$ (Assumption~\ref{as:space}.\ref{ass:norm_equiv}) and Eq.~\eqref{eq:Ph_Qh} we get
\begin{equation*} 
\dd{}{t} \norm{\theta(t)}_{h} \leq C \norm{(Q_h-P_h) \e^{tL}L\eta} \leq Ch^s \sum_{i=1}^{q+2}{\norm{L^i\eta}}.
\end{equation*}
Since additionally
\[ \norm{\theta(0)} \leq \norm{(Q_h-I)\eta} + \norm{\eta - \eta_h} \leq Ch^s \sum_{i=1}^{q+1}{\norm{L^i\eta}} + \norm{\eta - \eta_h},\]
we get
\[ \begin{split}
\norm{\theta(t)} &\leq C\norm{\theta(t)}_h \leq C(\norm{\theta(0)}_h + \int_0^t\dd{}{\tau}\norm{\theta(\tau)}_h\ \mathrm{d}\tau) \\
	&\leq C (\norm{\eta - \eta_h} + h^s\sum_{i=1}^{q+2}{\norm{L^i\eta}}), 
\end{split} \]
where the last inequality follows since we integrate over a bounded time interval. We thus arrive at the desired bound.
\end{proof}

\section{Full discretization}
\label{sec:full_discretization}
The full discretizations are defined by applying either the Douglas--Rachford scheme or the Peaceman--Rachford scheme to the ODE~\eqref{eq:discrete_evolution}. That is, to define the numerical flow $S_h$ replace all occurrences of $A$ and $B$ in equations \eqref{eq:DR} and \eqref{eq:PR} by $A_h$ and $B_h$, respectively. Then, the solution of the fully discretized evolution equation \eqref{eq:evolution} is given by $S_h^n\eta_h$. To bound the temporal error $\e^{tL_h}\eta_h - S_h^n\eta_h$ we need the stability bounds of Assumption~\ref{as:time}.
\begin{as} \label{as:time}\
\begin{enumerate}
	\item[\emph{DR.}] For the Douglas--Rachford scheme assume that
		\begin{equation*}
			\norm{A_hL_h^{-1}}_{h} \leq C, \quad \text{ for all } h\in (0,h_{max}].
		\end{equation*}
	\item[\emph{PR.}] For the Peaceman--Rachford scheme assume that
		\begin{equation*}
			\norm{A_hL_h^{-1}}_{h} \leq C \quad \text{and} \quad \norm{A_h^2L_h^{-2}}_{h} \leq C, \quad \text{ for all } h\in (0,h_{max}].
		\end{equation*}
\end{enumerate}
The constant $C$ is assumed to be independent of $h$.
\end{as}
\begin{rem}
Note that Assumption~\ref{as:time}.DR is equivalent to $\norm{B_hL_h^{-1}}_{h} \leq C$ and that Assumption~\ref{as:time}.PR implies that $\norm{A_hB_hL_h^{-2}}_h \leq C$, both uniformly in $h$.
\end{rem}

For the sake of completeness we give a short temporal convergence proof for the Douglas--Rachford splitting scheme. This also serves the purpose of clarifying why Assumption~\ref{as:time}.DR is needed. A slightly longer proof is given in \cite{Hansen2014}.
\begin{lem} \label{lem:DR} 
Let $S_h$ be given by~\eqref{eq:DR} in the manner described in the beginning of this section. If Assumptions \ref{as:space}.\ref{ass:discrete_dissipativity} and \ref{as:time}.DR are valid and $\eta_h \in \HH_h$, then the Douglas--Rachford splitting is first order convergent, that is
\[ \norm{\e^{nkL_h}\eta_h - S_h^n\eta_h}_h \leq C k \norm{L_h^2\eta_h}_h, \]
where $C$ can be chosen uniformly on bounded time intervals and, in particular, independently of $h$, $k$ and $n$.
\end{lem}

\begin{proof}
Define the operators
\[ a_h = kA_h, \quad b_h = kB_h, \quad \alpha_h = (I-a_h)\inv, \quad \beta_h = (I-b_h)\inv, \]
and note that
\begin{equation} \label{eq:identity_expansion}
I = \alpha_h - \alpha_h a_h = \alpha_h - a_h \alpha_h.
\end{equation}
We first expand the error using the telescopic sum
\begin{equation}
\e^{nkL_h}\eta_h - S_h^n\eta_h = \sum_{j=1}^n{S_h^{n-j}\beta_h (I-b_h) (\e^{jkL_h}\eta_h - S_h\e^{(j-1)kL_h}\eta_h)}. 
\label{eq:telescopic_sum}
\end{equation}
The operator $S_h^{n-j}\beta_h$ is nonexpansive on $(\HH_h, (\cdot,\cdot)_{h})$ which follows from the equality
\begin{equation}
S_h^{n-j}\beta_h = \beta_h (\alpha_h (I + a_h b_h)\beta_h)^{n-j} = \beta_h(\frac12\alpha_h(I+a_h)(I+b_h)\beta_h + \frac12 I)^{n-j}
\label{eq:DR_stability}
\end{equation}
and the fact that $(I+a_h)\alpha_h$ and $(I+b_h)\beta_h$ are nonexpansive. The latter holds as
\[ \begin{aligned}
\norm{(I+a_h) v_h}_h^2 &= \norm{v_h}_h^2 + 2(a_hv_h,v_h)_h  + \norm{a_hv_h}_h^2 \\
	&\leq \norm{v_h}_h^2 - 2(a_hv_h,v_h)_h  + \norm{a_hv_h}_h^2 = \norm{(I-a_h) v_h}_h^2.
\end{aligned} \]
By twice expanding the identity, we can rewrite the difference 
\[ \begin{split}
&(I-b_h) (\e^{jkL_h}\eta_h - S_h\e^{(j-1)kL_h}\eta_h)  \\
&\quad= (\alpha_h - \alpha_h a_h - (\alpha_h - a_h\alpha_h)b_h) \e^{jkL_h}\eta_h - \alpha_h(I + a_h b_h)\e^{(j-1)kL_h}\eta_h  \\
&\quad= \alpha_h(\e^{jkL_h}\eta_h - \e^{(j-1)kL_h}\eta_h - (a_h+b_h)\e^{jkL_h}\eta_h) \\
&\quad\quad+ a_h\alpha_h b_h (\e^{jkL_h}\eta_h - \e^{(j-1)kL_h}\eta_h) \\
&\quad= -k\alpha_h\int_{(j-1)k}^{jk}{\frac{\tau-(j-1)k}{k}\e^{\tau L_h}L_h^2\eta_h \ \mathrm{d}\tau} \\
&\quad\quad+ k a_h \alpha_h B_h L_h\inv \int_{(j-1)k}^{jk}{e^{\tau L_h}L_h^2\eta_h \ \mathrm{d}\tau}.
\end{split} \]
The operators $a_h \alpha_h$ and $B_h L_h\inv$ are uniformly bounded. The bound of the former follows from Eq.~\eqref{eq:identity_expansion} whereas the bound of the latter is a direct consequence of Assumption~\ref{as:time}.DR. Applying the $\norm{\cdot}_h$-norm to the error expansion \eqref{eq:telescopic_sum} and adding up the integrals yields the sought after error bound.
\end{proof}

Convergence for the Peaceman--Rachford scheme follows along the same lines, cf.\ \cite{Hansen2013}. We conclude the abstract analysis by proving convergence orders for the full discretization.
\begin{thm}\label{thm:abstract_error}
If Assumptions \ref{as:continuous} and \ref{as:space} are valid and $\eta \in \D(L^{q+r+1})$, then
\[ \norm{e^{nkL}\eta - S_h^nQ_h\eta} \leq C(h^s+k^r) \sum_{i=1}^{q+r+1} \norm{L^i\eta}, \]
under Assumption~\ref{as:time}.DR for the Douglas--Rachford scheme defined by~\eqref{eq:DR} and under Assumption~\ref{as:time}.PR for the Peaceman--Rachford scheme defined by~\eqref{eq:PR}. For the former scheme we have $r = 1$ and for the latter $r = 2$. The operator $Q_h$ is defined by Eq.~\eqref{eq:Q_h} and the constant $C$ can be chosen uniformly on bounded time intervals and, in particular, independently of $h, k$ and $n$.
\end{thm}

\begin{proof}
Define the operator $Z_h = L_h^{-(r+1)}P_hL^{r+1}:\D(L^{q+r+1})\rightarrow\HH_h$ and split the global error into three terms
\begin{equation} \begin{split}
\norm{e^{nkL}\eta - S_h^nQ_h\eta} &\leq \norm{\e^{nkL}\eta - \e^{nkL_h}Z_h\eta} + \norm{(\e^{nkL_h} - S_h^n)Z_h\eta} \\
	&+ \norm{S_h^n(Z_h\eta - Q_h\eta)}. 
\label{eq:error_separation}
\end{split} \end{equation}
The first term, the spatial error, can be bounded by Lemma~\ref{lem:space_error} as
\begin{equation}
\norm{\e^{nkL}\eta - \e^{nkL_h}Z_h\eta} \leq C(\norm{Z_h\eta-\eta} +  h^s \sum_{i=1}^{q+2}{\norm{L^i\eta}}),
\label{eq:term_1}
\end{equation}
where according to Eq.~\eqref{eq:I_Qh}
\begin{equation}
\norm{Z_h\eta-\eta} \leq \norm{Z_h\eta-Q_h\eta} + \norm{Q_h\eta - \eta} \leq C(\norm{Z_h\eta-Q_h\eta}_h + h^s\sum_{i=1}^{q+1}{\norm{L^i\eta}}).
\label{eq:term_1_2}
\end{equation}
For the second term of Eq.~\eqref{eq:error_separation}, the temporal error, we use the uniform equivalence of norms, Assumption~\ref{as:space}.\ref{ass:norm_equiv}, to perform the analysis on $(\HH_h, (\cdot,\cdot)_h)$. Under Assumption~\ref{as:space}.\ref{ass:discrete_dissipativity} and respective version of Assumption~\ref{as:time} we get from Lemma~\ref{lem:DR} respectively \cite[Theorem~2]{Hansen2013} that
\begin{equation}
\norm{(\e^{nkL_h} - S_h^n)Z_h\eta} \leq C\norm{(\e^{nkL_h} - S_h^n)Z_h\eta}_h \leq Ck^r \norm{L_h^{r+1} Z_h\eta}_h = Ck^r\norm{P_hL^{r+1}\eta}_h.
\label{eq:term_2_0}
\end{equation}
Further, from Assumptions \ref{as:space}.\ref{ass:norm_equiv} and \ref{as:space}.\ref{ass:projection} we get
\begin{equation} \label{eq:term_2} 
Ck^r\norm{P_hL^{r+1}\eta}_h \leq Ck^r \norm{P_h L^{r+1}\eta} \leq Ck^r \sum_{i=0}^{q}{\norm{L^{i+r+1}\eta}}.
\end{equation}
Considering the third term of Eq.~\eqref{eq:error_separation} we note that $S_h^n(I-\kappa B_h)\inv$ is nonexpansive on $(\HH_h,(\cdot,\cdot)_h)$. This follows by Eq.~\eqref{eq:DR_stability} for the Douglas--Rachford splitting with $\kappa = k$ and from \cite[Lemma~1]{Hansen2013} for the Peaceman--Rachford splitting with $\kappa = k/2$. Combining with the uniform equivalence of norms we get
\begin{equation} \begin{split}
\norm{S_h^n(Z_h\eta - Q_h\eta)} &\leq C\norm{S_h^n(Z_h\eta - Q_h\eta)}_h \\
	&= C\norm{S_h^n(I-\kappa B_h)\inv (I-\kappa B_h) (Z_h\eta - Q_h\eta)}_h \\
	&\leq C\norm{(I-\kappa B_h)(Z_h\eta - Q_h\eta)}_h.
\label{eq:term_3}
\end{split} \end{equation}
Then, the uniform bounds of $L_h\inv$ and $B_hL_h\inv$ together with the uniform equivalence of norms and the bound \eqref{eq:Ph_Qh} give
\begin{align} 
&\norm{(I-\kappa B_h)(Z_h\eta - Q_h\eta)}_h \leq \sum_{i=1}^{r}{\norm{(I-\kappa B_h)(L_h^{-(i+1)} P_h L^{i+1} - L_h^{-i} P_h L^{i})\eta}_h} \nonumber \\
	&\quad\leq C\sum_{i=1}^{r}{(\norm{L_h^{-i}}_h + \kappa \norm{B_hL_h\inv}_h\norm{L_h^{-i+1}}_h)\norm{(Q_h - P_h) L^{i}\eta}_h} \label{eq:term_4} \\
	&\quad\leq Ch^s \sum_{i=1}^{q+r+1}{\norm{L^i\eta}}. \nonumber
\end{align}
The term $\norm{Z_h\eta-Q_h\eta}_h$ of Eq.~\eqref{eq:term_1_2} can be bounded in the same manner. The theorem then follows by combining equations \eqref{eq:error_separation} -- \eqref{eq:term_4}.
\end{proof}

\begin{rem} \label{rem:proof_mod_without_invert}
If we remove the assumptions that $L$ and $L_h$ are invertible Lemma~\ref{lem:space_error}, Theorem~\ref{thm:abstract_error}, and their proofs need slight modifications to still hold. Additionally, modifications of Assumptions \ref{as:space}.\ref{ass:elliptic_error} and \ref{as:time} are needed. To this end replace all occurrences of $L\inv$ and $L_h\inv$ in these assumptions by $(I-L)\inv$ and $(I-L_h)\inv$ respectively. No assumption of $(I-L_h)\inv$ being uniformly bounded is needed since the resolvent is nonexpansive due to Assumption~\ref{as:space}.\ref{ass:discrete_dissipativity} and Remark~\ref{rem:L_h_dissipative}. Similarly, replace the operators $Q_h$ and $Z_h$ with $(I-L_h)\inv P_h (I-L_h)$ and $(I-L_h)^{-r-1} P_h (I-L_h)^{r+1}$, respectively. For Lemma~\ref{lem:space_error} consider the shifted evolution equation
\[ \dot w = (L-I)w \]
with solution $w(t) = \e^{-t} u(t)$. Since $L$ is maximal dissipative the operator $L-I$ is also maximal dissipative and additionally invertible. Using the modified assumptions the lemma follows for the flow of this shifted operator and thus also for the original operator through the simple relation between $u$ and $w$.
\end{rem}

\section{Dimension splitting and quadrature finite elements}
\label{sec:dimension_splitting}
We apply our convergence results to dimension splitting of the 2D diffusion problem defined by
\[ Lu = Au + Bu = \dedex (\lambda(x)\mu(y) \dedex u) + \dedey (\lambda(x)\mu(y) \dedey u), \]
with homogenous Dirichlet boundary conditions. For the spatial discretization we use a quadrature finite element method. This results in a discretization -- similar to finite difference discretizations -- where the matrices related to $A_h$ and $B_h$ decouple into block diagonal matrices with blocks corresponding to 1D problems. Therefore, the flow $S_h$ can be efficiently computed. 

Let $(\HH, (\cdot,\cdot)) = (L^2(\Omega), (\cdot,\cdot)_{L^2(\Omega)})$ where $\Omega = (0,1)^2$, additionally assume that $\lambda, \mu \in \C^2([0,1])$ and that $\lambda(x) \geq \lambda_0 > 0$, $\mu(y) \geq \mu_0 > 0$ for all $x,y \in (0,1)$. We define on $H^1_0(\Omega)$ the bounded and coercive bilinear form $b_L$ related to $L$ by
\begin{equation} \label{eq:bilinear_form}
(-Lv,\varphi) = b_L(v,\varphi) = (\lambda\mu \dedex v, \dedex \varphi) + (\lambda\mu \dedey v, \dedey \varphi),
\end{equation}
cf.\ \cite[Section~3.5]{Larsson}. In this context we can interpret $L$ as an unbounded, invertible and maximal dissipative operator on $L^2(\Omega)$ with domain
\begin{equation}
\D(L) = \{v \in H^1_0(\Omega);\ Lv \in L^2(\Omega)\} = H^2(\Omega) \cap H^1_0(\Omega),
\label{eq:elliptic_regularity_L2}
\end{equation}
for details see \cite[Sections~1--2]{Crouzeix} and \cite[Theorem~9.1.22]{Hackbusch}.

Consider the elliptic problem: Given an $f\in L^2(\Omega)$ find a $v \in H^1_0(\Omega)$ such that
\begin{equation}
b_L(v,\varphi) = (f,\varphi), \quad \text{ for all } \varphi \in H^1_0(\Omega).
\label{eq:weak_problem}
\end{equation} 
We note that in the above notation this is equivalent to solving $-Lv = f$. For Assumption~\ref{as:space}.\ref{ass:elliptic_error} we will later need the following regularity results and a priori estimates. Let $p \in [2, \infty)$ be fixed, then if $f \in L^p(\Omega)$ we have $v \in W^{2,p}(\Omega)$ and
\begin{equation}
\norm{v}_{W^{2,p}(\Omega)} \leq C (\norm{f}_{L^p(\Omega)} + \norm{v}_{W^{1,p}(\Omega)}). \label{eq:a_priori_Lp}
\end{equation}
Here $W^{2,p}(\Omega)$ denotes the space of functions whose weak derivatives up to order two are in $L^p(\Omega)$. The case $p = 2$ is considered in \cite[Theorem~9.1.22]{Hackbusch} and in \cite{Kadlec}. This result is used to characterize $\D(L)$ in Eq.~\eqref{eq:elliptic_regularity_L2}. Additionally, the term $\norm{v}_{H^1(\Omega)}$ can be bounded by $C \norm{f}_{L^2(\Omega)}$. For $p>2$ the a priori estimate~\eqref{eq:a_priori_Lp} follows from \cite[Theorem~4.3.2.4]{Grisvard} and the relation
\[ \norm{\Delta v}_{L^{p}(\Omega)} \leq C(\norm{f}_{L^{p}(\Omega)} + \norm{v}_{W^{1,p}(\Omega)}), \]
which holds for $v = -L\inv f$. The regularity result $v\in W^{2,p}(\Omega)$ is given by a slight modification of \cite[Theorem~4.4.3.7]{Grisvard}.

For the spatial discretization we construct continuous and quadrilateral finite element spaces. For a given $h \in (0,h_{\max}]$ such that $hM_h = 1$ and $M_h$ integer define the uniform square mesh $\{(x_i,y_j) = (ih,jh)\}_{i,j = 0}^{M_h}$. Let $K_{i,j} \subset \bar\Omega$ denote the square element defined as the convex hull of the mesh points $(x_i,y_j), (x_{i+1},y_j), (x_i,y_{j+1})$, and $(x_{i+1},y_{j+1})$. Denote by $\phi_{i,j}$ the continuous function which in each element is linear in $x$ and in $y$, takes the value 1 at $(x,y) = (x_i,y_j)$ and vanishes at all other mesh points. We then define $\HH_h$ as the linear span of $\{\phi_{i,j}\}_{i,j=1}^{M_h-1}$ and thus $\HH_h \subset H^1_0(\Omega)$. 

Using the trapezoidal rule on each element to approximate the $L^2(\Omega)$ inner product gives
\begin{equation}
(v,\varphi)_{h} = \frac{h^2}4 \sum_{i,j=0}^{M_h-1}{\sum_{I,J = 0}^1{(v\varphi)(x_{i+I},y_{j+J})}}.
\label{eq:inner_product_quadrature}
\end{equation}
for $v$ and $\varphi$ everywhere defined. See details in \cite[Sections~2.2~and~4.1]{Ciarlet}. By considering each element separately it is easy to verify that $(\cdot,\cdot)_{h}$ is an inner product on $\HH_h$ and that the induced norm is uniformly equivalent with the $L^2(\Omega)$-norm on $\HH_h$. Further, let $P_h$ be the the orthogonal projection with respect to $(\cdot,\cdot)_h$, defined on $\D(P_h) = \C(\bar\Omega)$. One easily realizes that $P_h$ coincides with the piecewise linear interpolation operator of \cite[Theorem~3.2.1]{Ciarlet}. Thus, since additionally $H^2(\Omega) \hookrightarrow \C(\bar\Omega)$, Assumption~\ref{as:space}.\ref{ass:projection} follows with $s = 2$ and $q = 1$ from this theorem and the a priori estimate~\eqref{eq:a_priori_Lp}. We note that for standard finite element schemes $P_h$ would be the normal projection from $L^2(\Omega)$ to $\HH_h$ and we would have $q=0$.

The discrete operator $L_h: \HH_h \rightarrow \HH_h$ and corresponding bilinear form $b_{L_h}$ on $\HH_h$ are defined by replacing $(\cdot,\cdot)$ with $(\cdot,\cdot)_h$ in Eq.~\eqref{eq:bilinear_form}, i.e.
\[ (-L_h v_h,\varphi_h)_h = b_{L_h}(v_h,\varphi_h) = (\lambda\mu \dedex v_h, \dedex \varphi_h)_h + (\lambda\mu \dedey v_h, \dedey \varphi_h)_h. \]
Similarly $A_h: \HH_h \rightarrow \HH_h$ and $B_h: \HH_h \rightarrow \HH_h$ are defined through the bilinear forms $b_{A_h}$ and $b_{B_h}$ given as the first respectively the second term in the right hand side of the above equation.
Note that extra care has to be given to element borders where the weak derivatives are not necessarily continuous. 

With the same analysis as for $L$ we can interpret $A_h, B_h$ and $L_h$ as maximal dissipative operators on $(\HH_h, (\cdot,\cdot)_h)$. The invertibility of $L_h$ and the uniform bound of this inverse follows as a direct consequence from the uniform ellipticity of $b_{L_h}$ (see \cite[Exercise~4.1.7]{Ciarlet} or \cite[Theorem~3]{Ciarlet1972}) and the uniform equivalence of $\norm{\cdot}$ and $\norm{\cdot}_h$. 

The discrete approximation of the elliptic problem~\eqref{eq:weak_problem} consists of finding a $v_h \in \HH_h$ such that
\begin{equation}
b_{L_h}(v_h,\varphi_h) = (f,\varphi_h)_h, \quad \text{for all } \varphi_h \in \HH_h,
\label{eq:discrete_problem}
\end{equation}
where $f\in L^2(\Omega)$ is assumed to be everywhere defined. Noting that $v = -L\inv f$ and $v_h = -L_h\inv P_h f$ one realizes that asserting Assumption~\ref{as:space}.\ref{ass:elliptic_error} is in this application equivalent to proving convergence of the discrete approximation~\eqref{eq:discrete_problem}. In \cite{Ciarlet1972} such results are given under the additional complication of curved boundaries. More precisely, \cite[Theorem~11]{Ciarlet1972} gives for $v\in W^{4,3}(\Omega) \cap H^1_0(\Omega)$ that
\[ \norm{v-v_h} \leq Ch^2 \norm{v}_{W^{4,3}(\Omega)}. \]
However in the current setting with straight boundaries the bound can be improved. Let $f\in H^2(\Omega) \hookrightarrow W^{1,3}(\Omega) \hookrightarrow L^{3}(\Omega)$, then $v \in W^{2,3}(\Omega)$ by the regularity of the elliptic equation \eqref{eq:weak_problem}. Following the proofs of \cite[Theorems~9~and~11]{Ciarlet1972} with some care we get
\[ \norm{v-v_h} \leq Ch^2 ( \norm{v}_{W^{2,3}(\Omega)} + \norm{f}_{H^{2}(\Omega)} ). \]
Using the a priori estimate~\eqref{eq:a_priori_Lp} first for $p=3$, then twice for $p=2$, we arrive at the assertion of Assumption~\ref{as:space}.\ref{ass:elliptic_error}:
\[ \begin{aligned}
\norm{L\inv f - L_h\inv P_h f} &\leq Ch^2 ( \norm{v}_{W^{2,3}(\Omega)} + \norm{f}_{H^{2}(\Omega)} ) \\
	&\leq Ch^2 ( \norm{f}_{L^3(\Omega)} + \norm{v}_{W^{1,3}(\Omega)} + \norm{f}_{H^{2}(\Omega)} ) \\
	&\leq Ch^2 ( \norm{f}_{H^2(\Omega)} + \norm{v}_{H^2(\Omega)} + \norm{f}_{H^{2}(\Omega)} ) \\
	&\leq Ch^2 ( \norm{f} + \norm{Lf}), \quad \text{ for all } f \in \D(L).
\end{aligned} \]

Finally we show the uniform bound in Assumption~\ref{as:time}.DR. To this end consider the symmetric and positive definite mass matrix $M$ and stiffness matrices $K_A$ and $K_B$ corresponding to the parabolic problems defined by $b_{A_h}$ and $b_{B_h}$. See~\cite[Section~10.1]{Larsson} for definitions. Due to the separable coefficient function $\lambda\mu$ the quadrature formula~\eqref{eq:inner_product_quadrature} gives stiffness matrices that can be written as Kronecker products
\[ \begin{split}
&K_A = K_{\lambda} \otimes D_{\mu}, \quad K_B = D_{\lambda} \otimes K_{\mu}, \quad \text{with} \\
&\begin{split} K_{\lambda} =\ &\text{tridiag}\big({-\lambda(x_{i-1})} - \lambda(x_i),\\
 &\lambda(x_{i-1}) + 2\lambda(x_i) + \lambda(x_{i+1}),\ -\lambda(x_i) - \lambda(x_{i+1}) \big)/2, \end{split} \\
&D_{\mu} = \text{diag}\big(\mu(y_j)\big), \quad i,j = 1, \dots, M_h - 1,
\end{split} \] 
and similarily for $K_{\mu}$ and $D_{\lambda}$. Additionally, $M = h^2I$ where $I$ is the identity matrix in $\R^{(M_h-1)^2}$. Let
\[ D = D_{\lambda} \otimes D_{\mu}, \]
then we have
\[ \begin{aligned}
&K_A(K_A+K_B)\inv = ((K_A+K_B)K_A\inv)\inv\\
&= \left(I + D_{\lambda}K_{\lambda}\inv \otimes K_{\mu}D_{\mu}\inv\right)\inv \\
&= \left(D^{1/2} \left( I + D_{\lambda}^{1/2}K_{\lambda}\inv D_{\lambda}^{1/2} \otimes D_{\mu}^{-1/2}K_{\mu}D_{\mu}^{-1/2} \right) D^{-1/2} \right)\inv \\
&= D^{1/2} \left( I + \left(D_{\lambda}^{-1/2}K_{\lambda}D_{\lambda}^{-1/2}\right)\inv \otimes D_{\mu}^{-1/2}K_{\mu}D_{\mu}^{-1/2} \right)\inv D^{-1/2},
\end{aligned} \]
where the Kronecker product in the middle factor of the last expression defines a symmetric and positive definite matrix. Due to the simple structure of the mass matrix we have
\[ \begin{aligned}
&\norm{A_hL_h\inv}_h = \norm{(-M\inv K_A)(-(K_A+K_B)\inv M)}_2\\
&= \norm{K_A(K_A+K_B)\inv}_2  = \norm{(I + D_{\lambda}K_{\lambda}\inv \otimes K_{\mu}D_{\mu}\inv)\inv}_2 \\
&\leq \frac{\sqrt{\norm{\lambda}_{L^\infty(0,1)}\norm{\mu}_{L^\infty(0,1)}}}{\sqrt{\lambda_0\mu_0}} \norm{ (I + (D_{\lambda}^{-1/2}K_{\lambda}D_{\lambda}^{-1/2})\inv \otimes D_{\mu}^{-1/2}K_{\mu}D_{\mu}^{-1/2} )\inv}_2 \\
&\leq \frac{\sqrt{\norm{\lambda}_{L^\infty(0,1)}\norm{\mu}_{L^\infty(0,1)}}}{\sqrt{\lambda_0\mu_0}} \leq C, \quad \text{ for all } h \in (0, h_{\text{max}}],
\end{aligned} \]
where $\norm{\cdot}_2$ denotes the Euclidean norm in $\R^{(M_h-1)^2}$. 

With all the relevant assumptions asserted we arrive at the following corollary of Theorem~\ref{thm:abstract_error} providing optimal convergence orders for the Douglas--Rachford dimension splitting combined with quadrature finite elements:
\begin{cor} \label{cor:DR}
Let $\HH$ and $L$ be defined as above and let $\HH_h$, $L_h$, $A_h$ and $B_h$ be given by the quadrature finite element method also defined above. Let the temporal discretization $S_h$ be given by the Douglas--Rachford scheme. Then, if $\eta \in \D(L^3)$,
\[ \norm{\e^{nkL}\eta - S_h^nQ_h\eta} \leq C (h^2 + k) \sum_{i = 1}^{3}{\norm{L^i\eta}}, \]
where $C$ can be chosen uniformly on bounded time intervals and, in particular, independently of $h, k$ and $n$.
\end{cor}

\section{Numerical experiments}
\label{sec:numerical_experiments}
With the help of the diffusion problem and spatial discretization discussed in Section~\ref{sec:dimension_splitting} we illustrate the convergence orders predicted by Theorem~\ref{thm:abstract_error} (and Corollary~\ref{cor:DR}). For our specific example we choose 
\[ \lambda(x) = x\sin(\pi x) + 0.1 \quad \text{and} \quad \mu(y) = \cos(2\pi y) + 1.1. \]
To assure that $\eta \in \D(L^4)$ let the initial value be given by
\[ \eta = \frac{L^{-4}\eta_0}{\norm{L^{-4}\eta_0}_{L^\infty(\Omega)}}, \]
where $\eta_0(x,y) = \sin(3\pi x)\cos(2\pi y)$.

To demonstrate the simultaneous convergence orders we find values of $k$ and $h$ such that the spatial and temporal errors are of approximately the same size. These parameters are then decreased keeping $h$ proportional to $\sqrt{k}$ for the Douglas--Rachford experiments and proportional to $k$ for the Peaceman--Rachford experiments. A reference solution is constructed by using a fine grid for the quadrature finite element method, $h = 2^{-10}$, and using the trapezoidal rule with step size $k = 2^{-13}$ as temporal discretization. The global error approximations are computed at time $t = 0.5$ in the $\norm{\cdot}_h$-norm, where $h = 2^{-10}$. The observed orders are presented in Figure~\ref{fig:convergence_DR_PR} and the results are in agreement with Theorem~\ref{thm:abstract_error}.
\begin{figure}[t] 
\centering \footnotesize
	\includegraphics[width=0.44\textwidth]{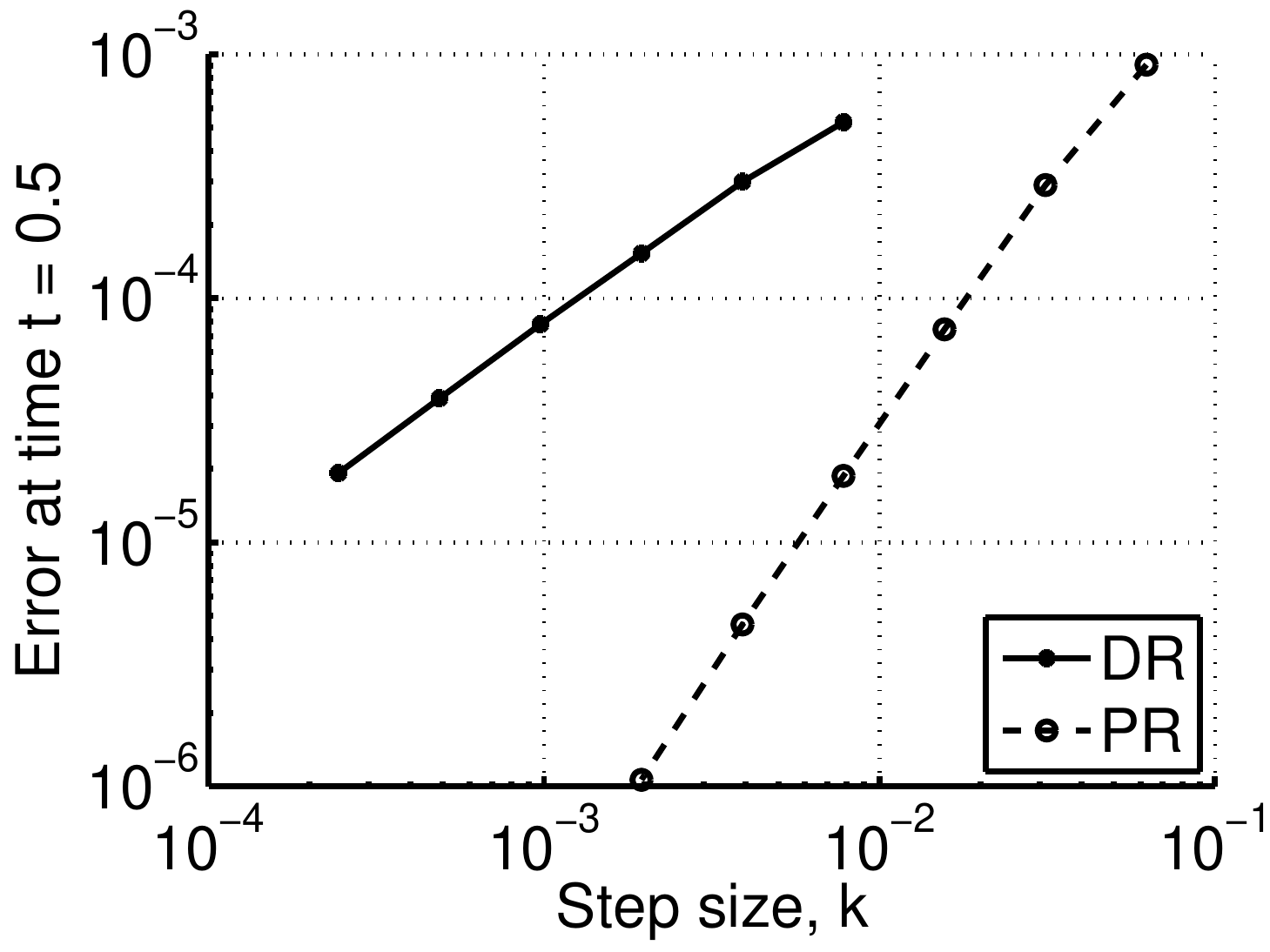}
	  \qquad
  \begin{tabular}[b]{l|ll|ll}
		 & \multicolumn{2}{|l|}{DR} & \multicolumn{2}{|l}{PR} \\
		\cline{2-3} \cline{4-5}
		$k$		& $h$		& Error	&	$h$		& Error \\
		\hline \hline
		1/16 &  &  & 1/16 & 9.1E-4 \\
		1/32 &  &  & 1/32 & 2.9E-4 \\
		1/64 &  &  & 1/64 & 7.5E-5 \\
		1/128 & 1/16 & 5.3E-4 & 1/128 & 1.9E-5 \\
		1/256 & 1/23 & 3.0E-4 & 1/256 & 4.6E-6 \\
		1/512 & 1/32 & 1.5E-4 & 1/512 & 1.1E-6 \\
		1/1024 & 1/45 & 7.8E-5 &  &  \\
		1/2048 & 1/64 & 3.9E-5 &  &  \\
		1/4096 & 1/91 & 1.9E-5 &  &  \vspace{1mm}
  \end{tabular}
  \caption{With $h$ proportional to $\sqrt{k}$ for the Douglas--Rachford (DR) experiments we observe first order convergence in $k$. Similarly, with $h$ proportional to $k$ for the Peaceman--Rachford (PR) experiments we observe second order convergence. The orders are in agreement with Theorem~\ref{thm:abstract_error}. The parameters values used in the experiments can be seen in the table together with the approximated global discretization errors.}
\label{fig:convergence_DR_PR}
\end{figure}

%%%% Acknowledgments %%%%%%%%
\section*{Acknowledgments}
The authors were supported by the Swedish Research Council under grant 621-2011-5588.

\bibliographystyle{abbrv}
%\bibliography{article_bib}
\bibliography{references}

%References should be cited in the text by a number in square brackets.
%Literature cited should appear on a separate page at the end of the article
%and should be styled and punctuated using standard abbreviations for journals
%(see Chemical Abstracts Service Source Index, 1989). For unpublished lectures of symposia,
%include title of paper, name of sponsoring society in full, and date.
%Give titles of unpublished reports with "(unpublished)" following the reference.
%Only articles that have been published or are in press should be included in the references.
%Unpublished results or personal communications should be cited as such in the text.
%Please note the sample at the end of this paper.
%
%%%%% Bibliography  %%%%%%%%%%
%\begin{thebibliography}{99}
%\bibitem{Berger}M. J. Berger and P. Collela, Local adaptive mesh refinement
%for shock hydrodynamics,
%J. Comput. Phys., 82 (1989), 62-84.
%\bibitem{deBoor}C. de Boor,  Good Approximation By Splines With Variable Knots II, in Springer Lecture
 %Notes Series 363, Springer-Verlag, Berlin, 1973.
%\bibitem{TanTZ} Z. J. Tan, T. Tang and Z. R. Zhang, A simple moving mesh method for one- and
%two-dimensional phase-field equations, J. Comput. Appl. Math., to appear.
%\bibitem{Toro}E. F. Toro, Riemann Solvers and Numerical Methods for Fluid Dynamics,
%Springer-Verlag Berlin Heidelbert, 1999.
%\end{thebibliography}

\end{document}